\documentclass[11pt]{amsart}

\usepackage{amsfonts, amssymb, amsmath, mdframed, graphicx, graphics, bbm, a4wide, indentfirst, tikz, pst-node, enumerate, comment}
\usetikzlibrary{hobby}

\usepackage[pagebackref, colorlinks=true, linktocpage=true, linkcolor=red, citecolor=red]{hyperref}

\newtheorem{Definition}{Definition}[section]

\newtheorem{Thm}{Theorem}
\newtheorem{Cor}[Thm]{Corollary}
\newtheorem{Corollary}[Definition]{Corollary}

\newtheorem{Proposition}[Definition]{Proposition}
\newtheorem{Lemma}[Definition]{Lemma}
\newtheorem{Remark}[Definition]{Remark}
\newtheorem*{Claim}{Claim}
\numberwithin{equation}{section}

\title{Stretched-exponential mixing for surface semiflows and Anosov flows}
\date{\today}
\subjclass[2020]{Primary: 37A25, 37C30; Secondary: 37D20.}
\keywords{Rate of mixing, Anosov flow, Hyperbolic flow, SRB measure}

\author[Daofei Zhang]{Daofei Zhang}
\address{Mathematics Institute, University of Warwick, Coventry, CV4 7AL, UK}
 \email{Daofei.Zhang@warwick.ac.uk}

\begin{document}

\begin{abstract}
For a surface semiflow that is a suspension of a \( C^{1+\alpha} \) expanding Markov interval map, we prove that, under the assumptions that the roof function is Lipschitz continuous and not cohomologous to a locally constant function, the semiflow exhibits stretched-exponential mixing with respect to the SRB measure. This result extends to hyperbolic skew-product semiflows and hyperbolic attractors. Specifically, codimension-one topologically mixing Anosov flows with Lipschitz continuous stable foliations demonstrate stretched-exponential mixing with respect to their SRB measures.
\end{abstract}

\maketitle

\section{Introduction}\label{sec 1}

The study of mixing rates for hyperbolic flows has a long history and is a fundamental problem in the ergodic theory of hyperbolic systems, following the seminal works of Sinai, Ruelle, and Bowen \cite{Bow75, Bow08, Rue76, Sin72}. Unlike hyperbolic diffeomorphisms, which exhibit exponential mixing with respect to Gibbs measures (the equilibrium states associated with H\"older potentials), the investigation of mixing rates for hyperbolic flows presents significant challenges. In \cite{Bow75}, Bowen and Ruelle posed the question of whether hyperbolic flows are exponentially mixing with respect to Gibbs measures. However, examples exist of hyperbolic flows that do not mix exponentially, with mixing occurring at arbitrarily slow rates \cite{Pol85, Rue83}. This realization prompted a more focused investigation into Anosov flows. Initially, results were sparse and often pertained to systems with additional algebraic structures. Notable examples include geodesic flows on hyperbolic surfaces \cite{Col84, Moo87, Rat87} and hyperbolic three-manifolds \cite{Pol92}. In these cases, representation theory provided the necessary tools for analysis, but the techniques used in these proofs do not generalize to all Anosov flows.

In \cite{Pol85}, Pollicott laid foundational work by establishing the connection between mixing rates and the spectral gaps of the associated complex transfer operator. Furthermore, in \cite{Pol84}, he pioneered a systematic study of the spectral properties of complex transfer operators. However, the results obtained in \cite{Pol84} fell short of proving the mixing rate, as the spectral gap was not quantified. A substantial breakthrough in this area began with Dolgopyat's seminal work in \cite{Dol98a}, which built on the pioneering work in \cite{Che98}. Dolgopyat provided a remarkable mechanism to ensure that the associated complex transfer operator has a uniform spectral gap, implying that the Anosov flow is exponentially mixing. The basic setting of this mechanism involves a smooth expanding map, a \(C^1\) roof function, and an associated Gibbs measure that satisfies a doubling property. This applies particularly to topologically mixing three-dimensional Anosov flows with smooth stable and unstable subbundles and all Gibbs measures.\footnote{The doubling property may not be satisfied for all Gibbs measures, as pointed out by Baladi and Vallée in \cite{Bal05}. Nonetheless, a weaker doubling property can still be proven, as shown in \cite{Dal21, Dal24}, which is sufficient to complete the proof.} Even without smooth stable and unstable subbundles, Dolgopyat managed to obtain a certain quantified spectral gap of the complex transfer operator, which can be bounded by the frequency of the Laplace transform of the correlation function \cite{Dol98a}. Consequently, any jointly non-integrable Anosov flow is rapidly mixing with respect to Gibbs measures. In \cite{Dol98b}, Dolgopyat also provided various conditions under which general hyperbolic flows exhibit rapid mixing with respect to Gibbs measures. For a more detailed introduction to rapid mixing and an additional condition that ensures rapid mixing, one can refer to \cite{Fie07,Mel18,Zha24}.

In the spirit of \cite{Dol98a}, significant progress has been made. In \cite{Liv04}, it was proven that Anosov flows preserving a smooth contact form are exponentially mixing with respect to the SRB measure, which is the Gibbs measure associated with the geometric potential. Generally, dealing with SRB measures is easier than with general Gibbs measures since SRB measures are absolutely continuous with respect to volume measures and thus satisfy the doubling property. In \cite{Giu13, Giu22}, the authors proved exponential mixing for a contact Anosov flow with a certain pinching condition with respect to the measure of maximal entropy, which is the Gibbs measure associated with the constant potential. The authors in \cite{Ara16a} generalized Dolgopyat's proof for Anosov flows to codimension-one hyperbolic attractors for SRB measures. The requirement for the regularity of stable and unstable foliations was relaxed in \cite{But20}. In \cite{Dal21, Dal24}, it was shown that for codimension-one hyperbolic attractors, Gibbs measures satisfying a certain weaker doubling property are sufficient to ensure exponential mixing. In \cite{Tsu23}, it was demonstrated that smooth Anosov flows are exponentially mixing with respect to Gibbs measures without any regularity assumption on stable or unstable foliations. The authors in \cite{Mel23} constructed a Young tower for any hyperbolic attractor, thereby proving exponential mixing for any hyperbolic attractor with a smooth stable foliation with respect to the SRB measure. In \cite{Gou19, Sto23}, it was demonstrated that contact Anosov flows are exponentially mixing with respect to regular Gibbs measures. This applies, in particular, to geodesic flows on three-dimensional manifolds of negative curvature and all Gibbs measures.

All of the aforementioned works reveal that the mixing rate of an Anosov flow is highly dependent on the oscillation of the temporal distance function. For instance, if the flow is jointly non-integrable, the temporal distance function does not vanish and exhibits small oscillations in a local area. Furthermore, its oscillation can be controlled by certain H\"older bounds, as the temporal distance function is H\"older continuous due to the stable and unstable subbundles being H\"older continuous. This property is sufficient to ensure that the flow is rapidly mixing \cite{Dol98a, Dol98b}. If we further assume that the temporal distance function has a globally high oscillation and can be controlled by certain smooth bounds—particularly in cases where the Anosov flow preserves a smooth contact form \cite{Liv04} or where the stable and unstable subbundles are smooth—then the flow enjoys exponential mixing \cite{Dol98a, Liv04}. Recently, it was proved in \cite{Tsu23} that if \(g_{t}\) is a \(C^{\infty}\) three-dimensional Anosov flow, then the temporal distance function can also be well controlled, resulting in exponential mixing.

The aim of this paper is to contribute to the study of mixing rates for flows and semiflows with uniform hyperbolicity. We focus on systems where the stable foliation (if it exists) and the unstable foliation have lower smoothness than \(C^{1}\). Specifically, we assume that the temporal distance function of an Anosov flow exhibits small oscillations in a local area and is less well-controlled by certain Lipschitz bounds. This scenario arises when the stable subbundle is Lipschitz continuous and the flow is jointly non-integrable. In this case, we prove that the Anosov flow exhibits stretched-exponential mixing with respect to the SRB measure (Corollary \ref{Stretched exponential mixing for Anosov flow}). Classically, we can reduce this scenario to that of a surface semiflow over an expanding Markov interval map. Under the assumption that the roof function is Lipschitz continuous and not cohomologous to a locally constant function, we show that the surface semiflow exhibits stretched-exponential mixing with respect to the SRB measure (Theorem \ref{Stretch exponential mixing for semiflow}).

This paper is organized as follows:

\begin{itemize}
\item In Section \ref{sec 2}, after reviewing foundational definitions, we state the main results more precisely. Subsection \ref{subsec 2.1} presents our findings regarding surface semiflows, specifically suspension semiflows over \(C^{1+\alpha}\) expanding Markov interval maps (refer to Theorem \ref{Stretch exponential mixing for semiflow}). Subsection \ref{subsec 2.2} applies these results to suspension semiflows over hyperbolic skew-products (see Corollary \ref{Stretch exponential mixing for hyperper skew product}). Finally, Subsection \ref{subsec 2.3} discusses results concerning Anosov flows and more general hyperbolic attractors (see Corollaries \ref{Stretched exponential mixing for hyperbolic attractor} and \ref{Stretched exponential mixing for Anosov flow}).
	
\item In Section \ref{sec 3}, following the classical approach pioneered by Pollicott \cite{Pol85}, we reduce Theorem \ref{Stretch exponential mixing for semiflow} to an estimate on a specific norm of the associated complex transfer operator \(\mathcal{L}_{s}\) and its resolvent (see Proposition \ref{Prop 3.4}). The proof of Theorem \ref{Stretch exponential mixing for semiflow} is provided at the end of this section. 
	
\item Section \ref{sec 4} establishes a Dolgopyat type estimate (Proposition \ref{Dolgopyat type estimate}) for the transfer operator \(\mathcal{L}_{ib}\). The main component of its proof is a cancellation lemma (Lemma \ref{Lemma 4.4}), which we address in the subsequent section. The estimate concerning \(\mathcal{L}_{s}\) as stated in Proposition \ref{Prop 3.4} is derived using a perturbation argument.
	
\item In Section \ref{sec 5}, we present the proof of Lemma \ref{Lemma 4.4}, leveraging the assumptions of Theorem \ref{Stretch exponential mixing for semiflow}. Since the roof function \(r\) is not smooth, we do not have the uniformly non-integrable (UNI) condition which is used by Dolgopyat to prove exponential mixing. However, under the assumption that $r$ is not cohomologous to locally constant functions, we establish a weaker UNI condition (Lemma \ref{Lemma 5.1}), applicable even when \(r\) is H\"older continuous. This stage yields a weak cancellation estimate on a small set, sufficient for proving rapid mixing. Further assuming \(r\) is Lipschitz continuous, along with Lemma \ref{Lemma 5.1} and domain connectivity, allows us to achieve a cancellation estimate on a large set with uniformly bounded Lebesgue measure, thereby completing the proof of Lemma \ref{Lemma 4.4}.
\end{itemize}

\noindent\textit{Acknowledgements}. I would like to thank my supervisor Mark Pollicott for his insightful suggestions and comments throughout the writing of this paper. The author is supported by the Warwick Mathematics Institute Centre for Doctoral Training and gratefully acknowledges funding from the ERC Advanced Grant 833802-Resonances.

\section{Statement of main results}\label{sec 2}

\subsection{Stretched-exponential mixing for surface semiflows}\label{subsec 2.1}

Let $\{I_{i}\}_{i=1}^{N}$ be finitely many pairwise disjoint bounded open intervals of $\mathbb{R}$. Denote $I:=\bigsqcup_{i}I_{i}$. Let $\alpha\in(0,1)$.	

\begin{Definition}
A surjection $T:I\to I$ is called an $C^{1+\alpha}$ expanding Markov interval map if it satisfies the following properties:
	\begin{itemize}
		\item Markov property: whenever $I_{i}\cap T^{-1}I_{j}\not=\emptyset$, we have $T$ is a $C^{1+\alpha}$ diffeomorphism from $I_{i}\cap T^{-1}I_{j}$ to $I_{j}$ and can be extended to the boundary. By this, we mean there exist an open interval $I^{\prime}$ which strictly contains $I_{i}\cap T^{-1}I_{j}$ and a $C^{1+\alpha}$ diffeomorphism $\widetilde{T}$ from $I^{\prime}$ to $\widetilde{T}I^{\prime}$ such that $\widetilde{T}|_{I_{i}\cap T^{-1}I_{j}}=T|_{I_{i}\cap T^{-1}I_{j}}$.
		\item Expansion: there exist $C>0$ and $\lambda\in(0,1)$ such that for any $n\in\mathbb{N}^{+}$, whenever $T^{n}(x)\in I$ we have $|\partial T^{n}(x)|\ge C\lambda^{-n}$, where $\partial$ denotes the differentiation.
	\end{itemize}
\end{Definition}

Given an expanding Markov interval map $T:I\to I$, we can associate it to an one-sided subshift of finite type. The related transition matrix $A$ is defined by 
$$
A(i,j)=\begin{cases}
	1,\quad & \text{if}\ I_{i}\cap T^{-1}I_{j}\not=\emptyset;\\
	0,\quad &\text{otherwise.}
\end{cases} 
$$
Therefore, we can view expanding Markov interval maps as a smooth version of one-sided subshifts. We always assume that \(T\) is topologically mixing, which is equivalent to \(A^{n} > 0\) for some \(n \in \mathbb{N}^{+}\). For convenience, we introduce some notation. For each \(1 \leq i \leq N\), we set \([i] := I_{i}\). For any \(n \in \mathbb{N}^{+}\) and each \(1 \leq i_{0}, \ldots, i_{n} \leq N\), when \([i_{0}] \cap \cdots \cap T^{-n}[i_{n}] \neq \emptyset\), we set \([i_{0} \cdots i_{n}] := [i_{0}] \cap \cdots \cap T^{-n}[i_{n}]\). By definition, for each \(n \in \mathbb{N}^{+}\) and \([i_{0} \cdots i_{n}]\), \(T^{n}\) is a diffeomorphism from \([i_{0} \cdots i_{n}]\) to \([i_{n}]\). We denote by \(\mathcal{H}_{n}\) the set of inverse branches of \(T^{n}\). In other words, an element \(y\) in \(\mathcal{H}_{n}\) has the form \(y := (T^{n} |_{[i_{0} \cdots i_{n}]})^{-1}\).

For $\alpha\in(0,1)$, let $C^{\alpha}(I)$ be the Banach space of $\alpha$-H\"older continuous functions on $I$ with the H\"older norm $||\cdot||_{\alpha}:=|\cdot|_{\infty}+|\cdot|_{\alpha}$, where $|\cdot|_{\alpha}$ denotes the usual $\alpha$ semi-norm. The Banach space of Lipschitz continuous functions on $I$ is denoted by $C^{Lip}(I)$ with the Lipschitz norm $||\cdot||_{\alpha}:=|\cdot|_{\infty}+|\cdot|_{Lip}$, where $|\cdot|_{Lip}$ presents the Lipschitz constant. We say a function $h\in C^{Lip}(I)$ is not cohomologous to locally constant functions if there is no $\chi$ which is constant on each $[ij]$ and $v\in C^{Lip}(I)$ such that $h=\chi+v\circ T-v$. Denote by $\varphi$ the geometric potential of $T$, namely, $\varphi(x)=-\log|\partial T(x)|$, which belongs to $C^{\alpha}(I,\mathbb{R})$ since $T$ is $C^{1+\alpha}$. We can associate $\varphi$ with a $T$-invariant probability measure $\mu_{\Phi}$ on $I$, which is absolutely continuous with respect to the normalized Lebesgue measure on $I$ and is called the SRB measure of $T$ \cite{Bal00}.

Our surface semiflows are suspension semiflows over  expanding Markov interval maps. Let $T:I\to I$ be a $C^{1+\alpha}$ expanding Markov interval maps, and let $r:I\to\mathbb{R}$ be a continuous function with $\inf_{x\in I}r(x)>0$. We define the suspension space as $I_{r}:=\{(x,u):0\le u\le r(x)\}/\sim$, where $(x,r(x))\sim(T(x),0)$. We also define the suspension semiflow $\phi_{t}$ by 
$$
\phi_{t}:I_{r}\to I_{r},\quad \phi_{t}(x,u)=(x,u+t),\quad t\ge0,
$$
with respect to $\sim$ on $I_{r}$. In the above definition, we call $\phi_{t}:I_{r}\to I_{r}$ the suspension semiflow $T$ under $r$. The SRB measure $\mu_{\varphi}$ induces a natural probability measure $\mu_{\varphi}\times\text{Leb}$ on $I_{r}$ by
$$
\int_{I_{r}}Fd(\mu_{\varphi}\times\text{Leb})=\dfrac{1}{\int_{I}rd\mu_{\varphi}}\int_{I}\int_{0}^{r(x)}F(x,u)dud\mu_{\varphi}(x),\quad F\in C(I_{r}).
$$ 
It can be shown that $\mu_{\varphi}\times\text{Leb}$ is $\phi_{t}$-invariant \cite{Par90}.

Given two test functions $E,F: I_{r}\to \mathbb{C}$, we are interested in the decay rate of their correlation function defined by
$$\rho_{E,F}(t):=\int_{I_{r}}E\circ\phi_{t}. Fd\mu_{\varphi}\times\text{Leb}-\int_{I_{r}}Ed\mu_{\varphi}\times\text{Leb}\int_{I_{r}}Fd\mu_{\varphi}\times\text{Leb}.$$
We say $\phi_{t}$ is mixing with respect to $\mu_{\varphi}\times\text{Leb}$ if $\rho_{E,F}(t)$ tends to $0$ as $t\to+\infty$ for any $E,F\in L^{2}(\mu_{\varphi}\times\text{Leb})$.

\begin{Definition}\label{Def 2.2}
For a complex-valued function $E$ on $I_{r}$, define $$|E|_{\alpha}:=\sup_{x\not=y}\sup_{ u\in[0,\min\{r(x),r(y)\}]}\frac{|E(x,u)-E(y,u)|}{|x-y|^{\alpha}}.$$
Denote by $|E|_{\infty}:=\sup_{(x,u)}|E(x,u)|$. Let $C^{\alpha,1}(I_{r})$ be the set of all complex-valued functions $E$ on $I_{r}$ which satisfy $||E||_{\alpha,1}:=|E|_{\infty}+|E|_{\alpha}+|\partial_{u}E|_{\infty}<\infty$ where $\partial_{u}$ denotes the differentiation along the $u$-coordinate.
\end{Definition}

In this paper, under the assumption that $r\in C^{Lip}(I,\mathbb{R})$ and is not cohomologous to locally constant functions, we provide a mechanism so that the surface semiflow $\phi_{t}:I_{r}\to I_{r}$ enjoys the following form of stretch-exponential mixing with respect to $\mu_{\varphi}\times\text{Leb}$.

\begin{Thm}\label{Stretch exponential mixing for semiflow}
Let $\phi_{t}:I_{r}\to I_{r}$ be the suspension semiflow of  a $C^{1+\alpha}$ expanding Markov interval map $T:I\to I$ under $r$. If $r\in C^{Lip}(I,\mathbb{R})$ and is not cohomologous to locally constant functions, then there exist $C>0$ and $\delta>0$ such that
$$|\rho_{E, F}(t)|\le C||E||_{\alpha,1}||F||_{\alpha,1}e^{-\delta\sqrt{t}},$$
for any $E,F\in C^{\alpha,1}(I_{r})$ and any $t>0$.
\end{Thm}

In \cite{Dol98a}, it is proven that if the roof function \( r \) is \( C^{1} \), then the surface semiflow \( \phi_{t}:I_{r}\to I_{r} \) is exponentially mixing with respect to \( \mu_{\varphi} \times \text{Leb} \), meaning that \( |\rho_{E, F}(t)| = O(e^{-\delta t}) \). Based on the methods discussed in this paper (see Remark \ref{Remark on rapid mixing} for a more detailed explanation), it can also be demonstrated that if the roof function \( r \) is only \( C^{\alpha} \) for some \( \alpha \in (0,1) \), then the surface semiflow \( \phi_{t}:I_{r}\to I_{r} \) is rapidly mixing with respect to \( \mu_{\varphi} \times \text{Leb} \), meaning that \( |\rho_{E, F}(t)| = O(t^{-n}) \) for any \( n \in \mathbb{N}^{+} \). These two result even hold for any Gibbs measure \cite{Dal24,Dol98a,Dol98b}. It is interesting to investigate whether the weaker Doubling property obtained in \cite{Dal21, Dal24} can be used to generalize Theorem \ref{Stretch exponential mixing for semiflow} to any Gibbs measure.

\subsection{Application to hyperbolic skew-product semiflows}\label{subsec 2.2}

Let $T:I\to I$ be a $C^{1+\alpha}$ expanding Markov interval map. Recall that $I=\bigsqcup_{i=1}^{N}I_{i}$. Given a finite number of metric spaces $\{S_{i}\}_{i=1}^{N}$, let $\widehat{I}_{i}=I_{i}\times S_{i}$ and $\widehat{I}=\bigsqcup_{i=1}^{N}\widehat{I}_{i}$.

\begin{Definition}
We say that a skew product $\widehat{T}:\widehat{I}\to\widehat{I}$ given by $\widehat{T}(x,y)=(T(x),G(x,y))$ is a $C^{1+\alpha}$  hyperbolic skew-product if 
\begin{itemize}
	\item there exist $C>0$ and $\lambda\in(0,1)$ such that $d(\widehat{T}^{n}(x,y),\widehat{T}^{n}(x,z))\le C\lambda^{n}d(y,z)$ for any $n\in\mathbb{N}^{+}$ and any $(x,y), (x,z)\in \widehat{I}$;
	\item the skew map $G$ is Lipschitz continuous. 
\end{itemize}
\end{Definition}

Given a $C^{1+\alpha}$ hyperbolic skew-product $\widehat{T}:\widehat{I}\to\widehat{I}$, let $\pi:\widehat{I}\to I$ be the coordinate project which is a semi-conjugacy between $\widehat{T}$ and $T$. The SRB measure $\mu_{\varphi}$ on $I$ induces a $\widehat{T}$-invariant probability measure $m$ on $\widehat{I}$ which satisfies $\pi^{*}m=\mu_{\varphi}$ \cite{But17}.

Given a roof function \( r \in C(I, \mathbb{R}) \), we can extend \( r \) to a function on \(\widehat{I}\) by defining \( r(x,y) = r(x) \). We then define the suspension semiflow \(\psi_{t}: \widehat{I}_{r} \to \widehat{I}_{r}\) of \(\widehat{T}\) under \( r \). This suspension semiflow \(\psi_{t}: \widehat{I}_{r} \to \widehat{I}_{r}\) is called a hyperbolic skew-product semiflow. Let \( m \times \text{Leb} \) be the \(\psi_{t}\)-invariant probability measure on \(\widehat{I}_{r}\) induced by \( m \). Similarly to the definition of \( C^{\alpha,1}(I_{r}) \) in Definition \ref{Def 2.2}, we can analogously define the space \( C^{\alpha,1}(\widehat{I}_{r}) \) of test functions on \(\widehat{I}_{r}\). As a consequence of Theorem \ref{Stretch exponential mixing for semiflow}, we have the following form of stretch-exponential mixing for hyperbolic skew-product semiflows.

\begin{Cor}\label{Stretch exponential mixing for hyperper skew product}
Let $\psi_{t}:\widehat{I}_{r}\to \widehat{I}_{r}$ be a hyperbolic skew-product semiflow as described above. If $r\in C^{Lip}(I,\mathbb{R})$ and is not cohomologous to locally constant functions, then there exist $C>0$ and $\delta>0$ such that
$$\bigg|\int_{\widehat{I}_{r}}E\circ\psi_{t}. Fdm\times\text{Leb}-\int_{\widehat{I}_{r}}Edm\times\text{Leb}\int_{\widehat{I}_{r}}Fdm\times\text{Leb}\bigg|\le C||E||_{\alpha,1}||F||_{\alpha,1}e^{-\delta\sqrt{t}},$$
for any $E,F\in C^{\alpha,1}(\widehat{I}_{r})$ and any $t>0$.
\end{Cor}
\begin{proof}
It is routine (see \cite{But20, Dal21}) to reduce the mixing rate of \(\psi_{t}:\widehat{I}_{r}\to \widehat{I}_{r}\) with respect to \(m \times \text{Leb}\) to that of the surface semiflow \(\phi_{t}:I_{r}\to I_{r}\) with respect to \(\mu_{\varphi} \times \text{Leb}\). The key property required is that the measure \(m\) on \(\widehat{I}\) can be disintegrated into conditional measures along local stable leaves \(\bigsqcup_{x \in I} \pi^{-1}(x)\), and that this disintegration is H\"older continuous. This property holds true if the skew map \(G\) is Lipschitz continuous \cite{But17}. Consequently, the correlation functions between \(\psi_{t}:\widehat{I}_{r}\to \widehat{I}_{r}\) with respect to \(m \times \text{Leb}\) and \(\phi_{t}:I_{r}\to I_{r}\) with respect to \(\mu_{\varphi} \times \text{Leb}\) are exponentially small, as shown in \cite[Equation (6)]{But20}. For more details, refer to \cite[page 2261-2262]{But20}. Thus, the desired result follows from Theorem \ref{Stretch exponential mixing for semiflow}.
\end{proof}

\subsection{Application to codimension-one hyperbolic attractors and Anosov flows}\label{subsec 2.3}

Let \( g_{t}:M \to M \) be a \( C^{2} \) flow on a smooth compact Riemannian manifold \( M \), and let \( TM \) be the tangent bundle of \( M \). Let \( \Lambda \subset M \) be a compact \( g_{t} \)-invariant set. The dynamical system we are interested in is the flow \( g_{t} \) acting on \( \Lambda \) with the following properties:
\begin{enumerate}[$(1).$]
	\item \( \Lambda \) consists of more than a single closed orbit, and the closed orbits of \( g_{t} \) in \( \Lambda \) are dense;
	
	\item the flow \( g_{t} \) acts transitively on \( \Lambda \), and there exists an open neighbourhood \( U \) of \( \Lambda \) in \( M \) such that \( \Lambda = \bigcup_{t>0} g_{t} U \);
	
	\item we can express \( T_{\Lambda} M = E^{c} \oplus E^{s} \oplus E^{u} \), where \( E^{c} \) is the line bundle of the flow, and \( E^{s} \) and \( E^{u} \) are subspaces that remain invariant under the differential \( Dg_{t} \), with \( E^{s} \) contracting and \( E^{u} \) expanding. Specifically, there exist constants \( C > 0 \) and \( \delta > 0 \) such that \( \|Dg_{t} v\| \le Ce^{-\delta t} \|v\| \) for all \( v \in E^{s} \) and \( t \ge 0 \), and \( \|Dg_{-t} v\| \le Ce^{-\delta t} \|v\| \) for all \( v \in E^{u} \) and \( t \ge 0 \).
\end{enumerate}
Under these settings, \( \Lambda \) is commonly referred to as a hyperbolic attractor of \( g_{t} \) \cite{Par90}. In the case where $\Lambda=M$, we call $g_{t}:M\to M$ an Anosov flow. We say $g_{t}:\Lambda\to\Lambda$ is jointly non-integrable if the subbundle $E^{s}\oplus E^{u}$ is not integrable. We call $g_{t}:\Lambda\to\Lambda$ codimension-one if $\dim(E^{s})=1$ or $\dim(E^{u})=1$.  For convenience, we consistently assume that a codimension-one hyperbolic attractor under consideration has $\dim(E^{u})=1$. Alternatively, we can consider the reversed vector field. 

For any Hölder continuous real-valued function $\Phi$ on $\Lambda$, we can associate with it a $g_t$-invariant probability measure $\mu_{\Phi}$ called the Gibbs measure of $\Phi$ \cite{Bow75}. In the case where $\Phi$ being the geometric potential defined by $\Phi(x)=\lim_{t\to0}\frac{1}{t}\log \det Dg_{t}|_{E^{u}(x)}$, the corresponding Gibbs measure $\mu_{\Phi}$ is called the Sinai-Ruelle-Bowen(SRB) measure. By Corollary \ref{Stretch exponential mixing for hyperper skew product}, we have the following form of stretch-exponential mixing for codimension-one hyperbolic attractors.

\begin{Cor}\label{Stretched exponential mixing for hyperbolic attractor}
Let $\Lambda$ be a codimension-one hyperbolic attractor of a $C^{2}$ flow $g_{t}:M\to M$. Assume $E^{s}$ is Lipschitz continuous and $g_{t}:\Lambda\to\Lambda$ is jointly non-integrable, then there exist $C>0$ and $\delta>0$ such that $$\bigg|\int_{\Lambda}E\circ g_{t}. Fd\mu_{\Phi}-\int_{\Lambda}Ed\mu_{\Phi}\int_{\Lambda}Fd\mu_{\Phi}\bigg|\le C||E||_{C^{1}}||F||_{C^{1}}e^{-\delta\sqrt{t}},$$
for any $E,F\in C^{1}(M)$ and any $t>0$.
\end{Cor}
\begin{proof}
It is also routine (see \cite{Dal21,Dal24}) to reduce the mixing rate of \(g_{t}:\Lambda\to \Lambda\) with respect to \(\mu_{\Phi}\) to that of a hyperbolic skew-product semiflow \(\psi_{t}:\widehat{I}_{r}\to \widehat{I}_{r}\) with respect to \(m \times \text{Leb}\). By \cite{Bow73, Rat73}, we can choose a Markov section $\{R_{i}\}_{i=1}^{N}$ of $g_{t}:\Lambda\to\Lambda$ of arbitrary small size, where these sets $R_{i}$ are suitable local cross-sections of the flow $g_{t}$. We can then regard $g_{t}:\Lambda\to\Lambda$ as the suspension flow $\psi_{t}$ of the first return map $P:R\to R$ on the section $R:=\cup_{i}R_{i}$ with the roof function being the first return time $r: R\to\mathbb{R}^{+}$. The assumption of $\Lambda$ is an attractor and $\dim(E^{u})=1$ ensures that the first return map $P:R\to R$ is a $C^{1+\alpha}$ hyperbolic skew-product. Since $E^{s}$ is Lipschitz continuous, we can choose these local cross-sections $R_{i}$ to be Lipschitz as well. Consequently, the skew map is Lipschitz continuous. Thus, the suspension flow $\psi_{t}$ is a hyperbolic skew-product semiflow. For more details, one can refer to \cite{Dal21}.
\end{proof}

In the case of Anosov flows, it is proven in \cite{Pla72} that joint non-integrability implies topological mixing, and they are equivalent if the flow is codimension-one. Thus, for topologically mixing codimension-one Anosov flows, our Corollary \ref{Stretched exponential mixing for hyperbolic attractor} can be stated more succinctly.

\begin{Cor}\label{Stretched exponential mixing for Anosov flow}
	Let $g_{t}:M\to M$ be a $C^{2}$ topologically mixing codimension-one Anosov flow. If $E^{s}$ is Lipschitz continuous, then $g_{t}$ enjoys the same form of stretched-exponential mixing with respect to the SRB measure as shown in Corollary \ref{Stretched exponential mixing for hyperbolic attractor}.
\end{Cor}

\section{Laplace transform and a Dolgopyat type estimate}\label{sec 3}

The strategy for proving Theorem \ref{Stretch exponential mixing for semiflow} is now classical. It involves investigating the Laplace transform of the correlation function. By demonstrating that the Laplace transform can be analytically extended to a specific region in the left half-plane, one can derive a decay rate for the correlation function, which depends on the shape of this region. The Laplace transform can be expressed as the resolvent of a complex transfer operator, and the analytical extension can be established through a certain norm estimate of the resolvent, as stated in Proposition \ref{Dolgopyat type estimate}. This section is dedicated to fulfilling this objective.

Recalling the correlation function for $E, F\in C^{\alpha,1}(I_{r})$ as
$$
\rho_{E,F}(t)=\int_{I_{r}} E\circ\phi_{t} Fd\mu_{\varphi}\times\text{Leb}-\int_{I_{r}}Ed\mu_{\varphi}\times\text{Leb}\int_{I_{r}} Fd\mu_{\varphi}\times\text{Leb}.
$$
As usual, by replacing $E$ with $E-\int_{I_{r}}Ed\mu_{\varphi}\times\text{Leb}$, we can assume $\int_{I_{r}}Ed\mu_{\varphi}\times\text{Leb}=\int_{I_{r}}Fd\mu_{\varphi}\times\text{Leb}=0$. It then more convenient to decompose $\rho_{E,F}(t)=\Delta_{E,F}(t)+\chi_{E,F}(t)$ where
$$\Delta_{E,F}(t):=\int_{I}\int_{0}^{\max\{0,r-t\}}E\circ\phi_{t}Fdud\mu_{\varphi}\quad\text{and}\quad\chi_{E,F}(t):=\int_{I}\int_{\max\{0,r-t\}}^{r}E\circ\phi_{t}Fdud\mu_{\varphi}.$$
The function $\Delta_{E,F}$ satisfies $\Delta_{E,F}(t)=0$ when $t\ge |r|_{\infty}$. Thus, $\rho_{E,F}$ and $\chi_{E,F}$ have the same asymptotic form. Consider the Laplace transform $\widehat{\chi}_{E,F}$ of $\chi_{E,F}$ defined as $\widehat{\chi}_{E,F}(s)=\int_{0}^{\infty}e^{-st}\chi_{E,F}(t)dt$. Since $\chi_{E,F}$ is bounded by $|E|_{\infty}|F|_{\infty}$, we have $\widehat{\chi}_{E,F}$ is analytic in the region $\{s=a+ib:a>0\}$. Denote by
$$e_{s}(x):=\int_{0}^{r(x)}e^{-su}E(x,u)du\quad\text{and}\quad f_{s}(x):=\int_{0}^{r(x)}e^{-su}F(x,u)du.$$
Obviously, $e_{s}$ and $f_{s}$ belong to $C^{\alpha}(I)$. Moreover, we can more precisely bound them as follows.

\begin{Lemma}\label{Lemma 3.1}
There exists $C_{1}>0$ such that for any $s=a+ib$ with $s\not=0$ and any $E,F\in C^{\alpha,1}(I_{r})$, we have
$$
\left\{\begin{aligned}
&|e_{s}|_{\infty}\le C_{1}e^{|a||r|_{\infty}}|s|^{-1}||E||_{\alpha,1}\\
&|f_{s}|_{\infty}\le C_{1}e^{|a||r|_{\infty}}|s|^{-1}||F||_{\alpha,1}
\end{aligned}
\right.\quad\text{and}\quad\left\{\begin{aligned}
	&|e_{s}|_{\alpha}\le C_{1}e^{|a||r|_{\infty}}||E||_{\alpha,1}\\
	&|f_{s}|_{\alpha}\le C_{1}e^{|a||r|_{\infty}}||F||_{\alpha,1}.
\end{aligned}
\right.
$$
\end{Lemma}
\begin{proof}
By integration by parts, we can write
$$e_{s}(x)=\frac{1}{-s}e^{-su}E(x,u)\bigg|_{0}^{r(x)}+\int_{0}^{r(x)}\frac{1}{s}e^{-su}\partial_{u}E(x,u)du,$$
which easily implies the first part. On the other hand, for any $x,y\in I$, without loss of generality we assume $r(x)\le r(y)$. Then, applying a basic triangle inequality yields 
$$|e_{s}(x)-e_{s}(y)|\le\int_{0}^{r(x)}e^{-su}|E(x,u)-E(y,u)|du+\int_{r(x)}^{r(y)}e^{-su}|E(y,u)|du,$$
which implies the second part.
\end{proof}

The following result is essentially due to Pollicott \cite{Pol85}.

\begin{Lemma}\label{Lemma 3.2}
For any $s=a+ib$ with $a>0$, we have $\widehat{\chi}_{E,F}(s)=\sum_{n=1}^{\infty}\int_{I}e^{-sr_{n}}e_{s}\circ T^{n} f_{-s}d\mu_{\varphi}$.
\end{Lemma}
\begin{proof}
By definition, we compute
\begin{equation*}
\begin{aligned}
	&\widehat{\chi}_{E,F}(s)\\
=&\int_{0}^{\infty}\int_{I}\int_{\max\{0,r(x)-t\}}^{r(x)}e^{-st}E(x,u+t)F(x,u)dud\mu_{\varphi}(x)dt\\
	=&\int_{I}\int_{0}^{r(x)}\int_{t\ge r(x)-u}e^{-st}E(x,u+t)F(x,u)dtdud\mu_{\varphi}(x)\\
			=&\int_{I}\int_{t^{\prime}\ge r(x)}e^{-st^{\prime}}E(x,t^{\prime})dt^{\prime}\int_{0}^{r(x)}e^{su}F(x,u)dud\mu_{\varphi}(x)\\
			=&\int_{I}\sum_{n=1}^{\infty}\bigg[\int_{r_{n}(x)}^{r_{n+1}(x)}e^{-st^{\prime}}E(x,t^{\prime})dt^{\prime}\bigg]\bigg[\int_{0}^{r(x)}e^{su}F(x,u)du\bigg]d\mu_{\varphi}(x)\\
			=&\sum_{n=1}^{\infty}\int_{I}e^{-sr_{n}}e_{s}\circ T^{n}f_{-s}d\mu_{\varphi},
		\end{aligned}
	\end{equation*}
which completes the proof.
\end{proof}

It is well-known \cite{Bal00} that the normalized Lebesgue measure on $I$ is the eigenmeasure with eigenvalue $1$ of the following transfer operator $\mathcal{L}_{\varphi}$,
$$\mathcal{L}_{\varphi}:C^{\alpha}(I)\to C^{\alpha}(I),\quad \mathcal{L}_{\varphi}h(x)=\sum_{y\in\mathcal{H}_{1}}e^{\varphi\circ y(x)}h\circ y(x),$$
where $\mathcal{H}_{1}$ represents the set of inverse branches of $T$. We also have $\mathcal{L}_{\varphi}h_{\varphi}=h_{\varphi}$ for some $h_{\varphi}\in C^{\alpha}(I)$ with $\inf_{x\in I}h(x)>0$. Furthermore, the SRB measure $\mu_{\varphi}$ is equal to $h_{\varphi}dx$ where $dx$ present the normalized Lebesgue measure on $I$. After adding a coboundary to $\varphi$, we can assume $\varphi$ is normalized, i.e., $\mathcal{L}_{\varphi}1=1$ and $\mathcal{L}_{\varphi}^{*}\mu_{\varphi}=\mu_{\varphi}$. In particular, by Lemma \ref{Lemma 3.2}, we have the following.

\begin{Corollary}\label{Cor 3.3}
For any $s=a+ib$ with $a>0$, we can further write
$\widehat{\chi}_{E,F}(s)=\int_{I}e_{s}\frac{\mathcal{L}_{s}}{1-\mathcal{L}_{s}}f_{-s}d\mu_{\varphi}$, where $\mathcal{L}_{s}$ denotes the transfer operator of $\varphi-sr$.
\end{Corollary}

We have the following estimate on $\mathcal{L}_{s}$ and its resolvent $(1-\mathcal{L}_{s})^{-1}$.

\begin{Proposition}\label{Prop 3.4}
There exist $C_{2}>0$, $C_{3}>0$ and $\delta_{1}>0$ such that for any $s=a+ib$ with $|a|\le\frac{\delta_{1}}{\log|b|}$ and $|b|\ge3$ and any $h\in C^{\alpha}(I)$, we have $$||\mathcal{L}_{s}h||_{b}\le C_{2}||h||_{b}\quad\text{and}\quad ||(1-\mathcal{L}_{s})^{-1}h||_{b}\le C_{2}\log|b|||h||_{b},$$
where $||h||_{b}=\max\{|h|_{\infty},|h|_{\alpha}/C_{3}|b|^{\alpha}\}$.
\end{Proposition}

The above proposition will be proved in the next section. Now, we are at last in a position to prove Theorem \ref{Stretch exponential mixing for semiflow}.

\begin{proof}[\textbf{Proof of Theorem \ref{Stretch exponential mixing for semiflow} assuming Proposition \ref{Prop 3.4}}]
Since $\rho_{E,F}(t)=\Delta_{E,F}(t)+\chi_{E,F}(t)$ with $\Delta_{E,F}(t)=0$ when $t\ge|r|_{\infty}$, we conclude that $\rho_{E,F}(t)$ and $\chi_{E,F}(t)$ have the same asymptotic form. Thus, it is sufficient to prove a decay rate of $\chi_{E,F}(t)$.

For any $E,F\in C^{\alpha,1}(I_{r})$, the analytic extension of $\widehat{\chi}_{E,F}$ is given by Corollary \ref{Cor 3.3}. Furthermore, by Lemma \ref{Lemma 3.1} and Proposition \ref{Prop 3.4}, we know that $\widehat{\chi}_{E,F}$ is analytic in the region of $s+a=ib$ with  $|a|\le\frac{\delta_{1}}{\log|b|}$ and $|b|\ge3$, and satisfies
\begin{equation}\label{3.1}
|\widehat{\chi}_{E,F}(s)|\le C_{4}\frac{\log|b|}{|b|^{1+\alpha}}||E||_{\alpha,1}||F||_{\alpha,1},
\end{equation}
for some uniform constant $C_{4}>0$. On the other hand, by an argument in \cite{Mel18}, the analyticity of $\widehat{\chi}_{E,F}$ in the region of $s=a+ib$ with  $|a|\le\frac{\delta_{1}}{\log3}$ and $|b|\le3$ can be derived from the weak-mixing property of $\phi_{t}$ or Proposition \ref{Prop 3.4} and it satisfies 
\begin{equation}\label{3.2}
|\widehat{\chi}_{E,F}(s)|\le C_{5}||E||_{\alpha,1}||F||_{\alpha,1},
\end{equation}
for some uniform constant $C_{5}>0$.
	
Now, together with \eqref{3.1} and \eqref{3.2}, we can apply the inverse Laplace formula to the curve $\Gamma=\{s=a+ib\ |\ a=-\min\{\frac{\delta_{1}}{\log|b|},\frac{\delta_{1}}{\log3}\},\ b\in\mathbb{R}\}$ to obtain that
\begin{equation}\label{3.3}
\begin{aligned}
&|\chi_{E,F}(t)|=\bigg|\int_{\Gamma}e^{st}\widehat{\chi}_{E,F}(s)ds\bigg|\\
\le&\bigg|\int_{|b|\ge3}e^{-\delta_{1}t/\log|b|}C_{4}\frac{\log|b|}{|b|^{1+\alpha}}||E||_{\alpha,1}||F||_{\alpha,1}db\bigg|+\bigg|\int_{|b|\le3}e^{-\delta_{1}t/\log3}C_{5}||E||_{\alpha,1}||F||_{\alpha,1}db\bigg|.
\end{aligned}
\end{equation}
On the one hand, we have the trivial bound
\begin{equation}\label{3.4}
\bigg|\int_{|b|\le3}e^{-\delta_{1}t/\log3}C_{5}||E||_{\alpha,1}||F||_{\alpha,1}db\bigg|\le 6C_{5}||E||_{\alpha,1}||F||_{\alpha,1}e^{-\delta_{1}t/\log3}.
\end{equation}
On the other hand, it is not difficult to show that there exist $C_{6}>0$ and $\delta_{2}>0$ such that for any $t>0$,
\begin{equation}\label{3.5}
\bigg|\int_{|b|\ge3}e^{-\delta_{1}t/\log|b|}\frac{\log|b|}{|b|^{1+\alpha}}db\bigg|\le C_{6}e^{-\delta_{2}\sqrt{t}}.
\end{equation}
Substituting \eqref{3.4} and \eqref{3.5} into \eqref{3.3}, we have 
\begin{equation}\label{3.6}
	|\chi_{E,F}(t)|\le C_{7}e^{-\delta_{2}\sqrt{t}}||E||_{\alpha,1}||F||_{\alpha,1},
	\end{equation}
for any $t\ge0$ and some uniform constant $C_{7}>0$. 

Finally, since $\rho_{E,F}(t)=\Delta_{E,F}(t)+\chi_{E,F}(t)$ with $\Delta_{E,F}(t)=0$ when $t\ge|r|_{\infty}$, we conclude that $\rho_{E,F}(t)$ has the same asymptotic form in \eqref{3.6} and thus completing the proof.
\end{proof}

\section{Proof of Proposition \ref{Prop 3.4}}\label{sec 4}

In this section, we prove Proposition \ref{Prop 3.4}. The key technical result is a Dolgopyat type estimate (Proposition \ref{Dolgopyat type estimate}) for the transfer operator \(\mathcal{L}_{ib}\). The main ingredient in the proof of Proposition \ref{Dolgopyat type estimate} is a cancellation lemma (Lemma \ref{Lemma 4.4}), which we address in the subsequent section. The estimate concerning \(\mathcal{L}_{s}\) and its resolvent \((1 - \mathcal{L}_{s})^{-1}\), as stated in Proposition \ref{Prop 3.4}, is derived from the Dolgopyat-type estimate for \(\mathcal{L}_{ib}\) in Proposition \ref{Dolgopyat type estimate} using a perturbation argument.

Since $T:I \to I$ is a $C^{1+\alpha}$ expanding Markov map, there exist $C_{8}>0$ and $\lambda\in(0,1)$ such that for any $n\in\mathbb{N}^{+}$, any $y\in\mathcal{H}_{n}$ and any $x,x^{\prime}\in I$,
\begin{equation}\label{4.1}
	|y(x)-y(x^{\prime})|\le C_{8}\lambda^{n}|x-x^{\prime}|,
\end{equation}
where $\mathcal{H}_{n}$ is the collection of inverse branches of $T^{n}$. Based on the above bound, it is straightforward to demonstrate
\begin{equation}\label{4.2}
	\sup_{n\in\mathbb{N}}\sup_{y\in\mathcal{H}_{n}}|r_{n}\circ y|_{Lip}<\infty\quad\text{and}\quad\sup_{n\in\mathbb{N}}\sup_{y\in\mathcal{H}_{n}}|\varphi_{n}\circ y|_{\alpha}<\infty.
\end{equation} 
These bounds can be used to prove the following Lasota-Yorke inequality.

\begin{Lemma}\label{Lemma 4.1}
	There exists $C_{9}>0$ such that for any $|b|\ge3$, any $h\in C^{\alpha}(I)$ and any $n\in\mathbb{N}^{+}$,
	$$|\mathcal{L}_{ib}^{n}h|_{\alpha}\le C_{9}|b|^{\alpha}|h|_{\infty}+C_{9}\lambda^{\alpha n}|h|_{\alpha}.$$
\end{Lemma}
\begin{proof}
Using \eqref{4.1} and \eqref{4.2}, and noting that $\mathcal{L}_{\varphi}1=1$, for any $x, x^{\prime}\in I$, we compute,
$$
\begin{aligned}
	&|\mathcal{L}_{ib}^{n}h(x)-\mathcal{L}_{ib}^{n}h(x^{\prime})|\le\sum_{y\in\mathcal{H}_{n}}|e^{(\varphi-ibr)_{n}\circ y(x)}h\circ y(x)-e^{(\varphi-ibr)_{n}\circ y(x^{\prime})}h\circ y(x^{\prime})|\\
	\le&\sum_{y\in\mathcal{H}_{n}}\big|e^{\varphi_{n}\circ y(x)}-e^{\varphi_{n}\circ y(x^{\prime})}\big||h|_{\infty}+e^{\varphi_{n}\circ y(x^{\prime})}\big|e^{ibr_{n}\circ y(x)}-e^{ibr_{n}\circ y(x^{\prime})}\big||h|_{\infty}\\
	&\quad \quad+e^{\varphi_{n}\circ y(x^{\prime})}|h\circ y(x)-h\circ y(x^{\prime})|\\
	\le&\sum_{y\in\mathcal{H}_{n}}e^{\varphi_{n}\circ y(x)+C_{9}^{\prime}}C_{9}^{\prime}|x-x^{\prime}|^{\alpha}|h|_{\infty}\\
	&\quad \quad+e^{\varphi_{n}\circ y(x^{\prime})}2\min\{|br_{n}\circ y(x)-br_{n}\circ y(x^{\prime})|,1\}|h|_{\infty}+e^{\varphi_{n}\circ y(x^{\prime})}|h|_{\alpha}C_{8}^{\alpha}\lambda^{\alpha n}|x-x^{\prime}|^{\alpha}\\
	\le&e^{C_{9}^{\prime}}C_{9}^{\prime}|x-x^{\prime}|^{\alpha}|h|_{\infty}+2|b|^{\alpha}C_{9}^{\prime}|x-x^{\prime}|^{\alpha}|h|_{\infty}+|h|_{\alpha}C_{8}^{\alpha}\lambda^{\alpha n}|x-x^{\prime}|^{\alpha},
\end{aligned}
$$
for some uniform constant $C_{9}^{\prime}>0$. The result then follows by choosing a suitable constant $C_{9}>C_{9}^{\prime}$.
\end{proof}

Recalling that \(||h||_{b}=\max\{|h|_{\infty},|h|_{\alpha}/C_{3}|b|^{\alpha}\}\), we choose \(C_{3}>0\) such that \(C_{9}/C_{3}<1/2\).

\begin{Corollary}\label{Cor 4.2}
	For any $|b|\ge3$, any \(n\in\mathbb{N}^{+}\), and any \(h\in C^{\alpha}(I)\), we have $\frac{|\mathcal{L}_{ib}^{n}h|_{\alpha}}{C_{3}|b|^{\alpha}}\le(1/2+C_{9}\lambda^{\alpha n})||h||_{b}$ and \(||\mathcal{L}_{ib}^{n}h||_{b}\le\max\{1, (1/2+C_{9}\lambda^{\alpha n})\}||h||_{b}\).
\end{Corollary}
\begin{proof}
	By Lemma \ref{Lemma 4.1}, 
	\begin{equation*}
		\dfrac{|\mathcal{L}_{ib}^{n}h|_{\alpha}}{C_{3}|b|^{\alpha}}\le \dfrac{C_{9}}{C_{3}}|h|_{\infty}+\dfrac{C_{9}\lambda^{\alpha n}}{C_{3}|b|^{\alpha}}|h|_{\alpha}.
	\end{equation*}
	In particular,
	\begin{equation*}
		\dfrac{|\mathcal{L}_{ib}^{n}h|_{\alpha}}{C_{3}|b|^{\alpha}}\le(1/2+C_{9}\lambda^{\alpha n})||h||_{b}.
	\end{equation*}
	Note that \(|\mathcal{L}_{ib}^{n}h|_{\infty}\le|h|_{\infty}\). Thus, we have \(||\mathcal{L}_{ib}^{n}h||_{b}\le\max\{1, (1/2+C_{9}\lambda^{\alpha n})\}||h||_{b}\) which completes the proof.
\end{proof}

We choose $N_{1}\in\mathbb{N}^{+}$ large enough such that for any $n\ge N_{1}$ we have $1/2+C_{9}\lambda^{\alpha n}\le 3/4$. We first deal with the easy case where \(h\in C^{\alpha}(I)\) satisfies $|h|_{\alpha}\ge 2C_{3}|b|^{\alpha}|h|_{\infty}$ as follows.

\begin{Lemma}\label{Lemma 4.3}
	For any $|b|\ge3$, any \(n\ge N_{1}\), and any \(h\in C^{\alpha}(I)\) with $|h|_{\alpha}\ge 2C_{3}|b|^{\alpha}|h|_{\infty}$, we have $|\mathcal{L}_{ib}^{n}h||_{b}\le\frac{3}{4}||h||_{b}$.
\end{Lemma}
\begin{proof}
	On the one hand, $|\mathcal{L}_{ib}^{n}h|_{\infty}\le|h|_{\infty}\le|h|_{\alpha}/2C_{3}|b|^{\alpha}\le\frac{1}{2}||h||_{b}$. On the other hand, by Corollary \ref{Cor 4.2}, we have $\frac{|\mathcal{L}_{ib}^{n}h|_{\alpha}}{C_{3}|b|^{\alpha}}\le(1/2+C_{9}\lambda^{\alpha n})||h||_{b}\le\frac{3}{4}||h||_{b}$. Thus, by definition, we have $|\mathcal{L}_{ib}^{n}h||_{b}\le\frac{3}{4}||h||_{b}$.
\end{proof}

It remains to deal with the more difficult situation where \(h\in C^{\alpha}(I)\) satisfies $|h|_{\alpha}\le 2C_{3}|b|^{\alpha}|h|_{\infty}$. To this end, we will use the assumption of Theorem \ref{Stretch exponential mixing for semiflow} to obtain the following cancellation of terms in a transfer operator.

\begin{Lemma}\label{Lemma 4.4}
There exist $N_{2}\in\mathbb{N}^{+}$, $\delta_{3}>0$ and $\delta_{4}\in(0,1)$ such that for any $|b|\ge3$ and any \(h\in C^{\alpha}(I)\) with $|h|_{\alpha}\le 2C_{3}b|^{\alpha}|h|_{\infty}$, there exists a subset $B\subset I$ with $\mu(B)\ge\delta_{3}$ such that $|\mathcal{L}_{ib}^{N_{2}}h(x)|\le(1-\delta_{4})|h|_{\infty}$ for any $x\in B$.
\end{Lemma}

The above lemma will be proved in the next section. As a corollary of Lemma \ref{Lemma 4.4}, we can obtain a cancellation the oscillatory integral $\int_{I}|\mathcal{L}_{ib}^{N_{2}}h|d\mu_{\varphi}$ as follows:
\begin{equation}\label{4.3}
	\begin{aligned}
		\int_{I}|\mathcal{L}_{ib}^{N_{2}}h|d\mu_{\varphi}=&\int_{B}|\mathcal{L}_{ib}^{N_{2}}h|d\mu_{\varphi}+\int_{I-B}|\mathcal{L}_{ib}^{N_{2}}h|d\mu_{\varphi}\\
		\le&(1-\delta_{4})|h|_{\infty}\mu_{\varphi}(B)+|h|_{\infty}\mu_{\varphi}(I-B)\\
		=&(1-\delta_{4}\mu_{\varphi}(B))|h|_{\infty}\le(1-\delta_{4}\delta_{3})|h|_{\infty}.
	\end{aligned}
\end{equation}
We can strengthen the above $L^{1}$ contraction to a $C^{0}$ contraction as follows.

\begin{Lemma}\label{Lemma 4.5}
There exists $C_{10}>0$ such that for any $|b|\ge3$ and any $h\in C^{\alpha}(I)$ with $|h|_{\alpha}\le 2C_{3}|b|^{\alpha}|h|_{\infty}$,
	$$|\mathcal{L}_{ib}^{C_{10}\log|b|+N_{2}}h|_{\infty}\le(1-2^{-1}\delta_{4}\delta_{3})|h|_{\infty}.$$
\end{Lemma}
\begin{proof}
Since $\mathcal{L}_{\varphi}$ acting on $C^{\alpha}(I)$ has a spectral gap \cite{Bal00}, there exist $C_{10}^{\prime}>0$ and $\delta_{5}\in(0,1)$ such that for any $h\in C^{\alpha}(I)$ and any $n\in\mathbb{N}$,
\begin{equation}\label{4.4}
||\mathcal{L}_{\varphi}^{n}h||_{\alpha}\le\int_{I}|h|d\mu_{\varphi}+C_{10}^{\prime}\delta_{5}^{n}||h||_{\alpha}.
\end{equation}
Now, by \eqref{4.3}, \eqref{4.4} and Corollary \ref{Cor 4.2}, for some uniform constant $C_{10}>0$ and any $|b|\ge3$,
	$$
	\begin{aligned}
		&|\mathcal{L}_{ib}^{C_{10}\log|b|+N_{2}}h|_{\infty}\le|\mathcal{L}_{\varphi}^{C_{10}\log|b|}|\mathcal{L}_{ib}^{N_{2}}h||_{\infty}\\
		\le&\int_{I}|\mathcal{L}_{ib}^{N_{2}}h|d\mu_{\varphi}+C_{10}^{\prime}\delta_{5}^{C_{10}\log|b|}||\mathcal{L}_{ib}^{N_{2}}h||_{\alpha}\\
		\le&(1-\delta_{4}\delta_{3})|h|_{\infty}+2^{-1}\delta_{4}\delta_{3}|h|_{\infty}\\
		=&(1-2^{-1}\delta_{4}\delta_{3})|h|_{\infty}
	\end{aligned}
	$$
	which completes the proof.
\end{proof}

The above $C^{0}$ contraction implies the following $||\cdot||_{b}$ contraction.

\begin{Corollary}\label{Cor 4.6}
For any $|b|\ge3$, and any $h\in C^{\alpha}(I)$ with $|h|_{\alpha}\le 2C_{3}|b|^{\alpha}|h|_{\infty}$,
	$$||\mathcal{L}_{ib}^{N_{2}+C_{10}\log|b|}h||_{b}\le(1-2^{-1}\delta_{4}\delta_{3})||h||_{b}.$$
\end{Corollary}
\begin{proof}
	This comes from Corollary \ref{Cor 4.2} and Lemma \ref{Lemma 4.5}.
\end{proof}

\begin{Proposition}\label{Dolgopyat type estimate}
There exists $C_{11}>0$ such that for any $|b|\ge3$ and any $h\in C^{\alpha}(I)$, we have $||\mathcal{L}_{ib}^{C_{11}\log|b|}h||_{b}\le(1-2^{-1}\delta_{4}\delta_{3})||h||_{b}$.
\end{Proposition}
\begin{proof}
This comes from Corollary \ref{Cor 4.2}, Lemma \ref{Lemma 4.3} and Corollary \ref{Cor 4.6}.
\end{proof}

\begin{proof}[\textbf{Proof of Proposition \ref{Prop 3.4}}]
Similarly to the proof of Lemma \ref{Lemma 4.1}, one can prove that there exists $C_{12}^{\prime}>0$ such that for any $s=a+ib$ and any $h\in C^{\alpha}(I)$,
\begin{equation*}
|(\mathcal{L}_{s}-\mathcal{L}_{ib})h|_{\infty}\le C_{12}^{\prime}|a||h|_{\infty}\quad\text{and}\quad |(\mathcal{L}_{s}-\mathcal{L}_{ib})h|_{\alpha}\le C_{12}^{\prime}|a||b|^{\alpha}|h|_{\infty}+C_{12}^{\prime}|a||h|_{\alpha}.
\end{equation*}
In particular, from the above, we deduce that for any $s=a+ib$ with $|b|\ge3$ and any $h\in C^{\alpha}(I)$,
\begin{equation}\label{4.5}
||(\mathcal{L}_{s}-\mathcal{L}_{ib})h||_{b}\le C_{12}^{\prime}(C_{3}^{-1}+1)|a|||h||_{b}.
\end{equation}
To prove the second part of Proposition \ref{Prop 3.4}, by Proposition \ref{Dolgopyat type estimate} and Corollary \ref{Cor 4.2}, for any $|b|\ge3$ and any $h\in C^{\alpha}(I)$ we have
\begin{equation}\label{4.6}
	\begin{aligned}
||(1-\mathcal{L}_{ib})^{-1}h||_{b}\le&\sum_{n=0}^{\infty}||\mathcal{L}_{ib}^{n}h||_{b}=\sum_{k=0}^{\infty}\sum_{j=0}^{C_{11}\log|b|-1}||\mathcal{L}_{ib}^{k\log|b|+j}h||_{b}\\
\le&\sum_{k=0}^{\infty}\sum_{j=0}^{C_{11}\log|b|-1}2C_{9}||\mathcal{L}_{ib}^{k\log|b|}h||_{b}\le\sum_{k=0}^{\infty}C_{11}2C_{9}\log|b|||\mathcal{L}_{ib}^{k\log|b|}h||_{b}\\
\le&C_{12}\log|b|||h||_{b},
\end{aligned}
\end{equation}
for some uniform constant $C_{12}>0$. In particular, by \eqref{4.5} and \eqref{4.6}, provided $|a|\le\frac{\delta_{1}}{\log|b|}$ where $\delta_{1}$ is small sufficient, we deduce that $||\frac{\mathcal{L}_{s}-\mathcal{L}_{ib}}{1-\mathcal{L}_{ib}}||_{b}\le 1/2$.  Now, we can express
\begin{equation*}
	\dfrac{1}{1-\mathcal{L}_{s}}=\dfrac{1}{1-\frac{\mathcal{L}_{s}-\mathcal{L}_{ib}}{1-\mathcal{L}_{ib}}}\dfrac{1}{1-\mathcal{L}_{ib}}.
\end{equation*}
Therefore, by \eqref{4.6} and $||\frac{\mathcal{L}_{s}-\mathcal{L}_{ib}}{1-\mathcal{L}_{ib}}||_{b}\le 1/2$, we have $||(1-\mathcal{L}_{s})^{-1}h||_{b}\le 2C_{12}\log|b| ||h||_{b}$, which proves the second part of Proposition \ref{Prop 3.4}. Finally, the proof of the first part of Proposition \ref{Prop 3.4} follows from Corollary \ref{Cor 4.2} and \eqref{4.5}.
\end{proof}

\section{Proof of Lemma \ref{Lemma 4.4}}\label{sec 5}

In this section, we present the proof of Lemma \ref{Lemma 4.4}, leveraging the assumptions of Theorem \ref{Stretch exponential mixing for semiflow}. Due to the non-smoothness of the roof function \(r\), we lack the uniformly non-integrable (UNI) condition crucial in Dolgopyat's proof of exponential mixing. However, assuming that \(r\) is not cohomologous to locally constant functions, we establish a weaker UNI condition (Lemma \ref{Lemma 5.1}), applicable even when \(r\) is H\"older continuous. This step provides a modest cancellation estimate on a small set, which is adequate for demonstrating rapid mixing. Moreover, under the additional assumption that \(r\) is Lipschitz continuous and considering Lemma \ref{Lemma 5.1} in conjunction with the connectivity of the domain \(I\), we attain a cancellation estimate over a larger set with uniformly bounded Lebesgue measure. This progression facilitates the straightforward derivation of the conclusion of Lemma \ref{Lemma 4.4}.

Recalling, $I=\bigsqcup_{i=1}^{N}I_{i}$ and for each \(1 \leq i \leq N\) we set \([i] := I_{i}\). For any \(n \in \mathbb{N}^{+}\) and each \(1 \leq i_{0}, \ldots, i_{n} \leq N\), when \([i_{0}] \cap \cdots \cap T^{-n}[i_{n}] \neq \emptyset\), we set \([i_{0} \cdots i_{n}] := [i_{0}] \cap \cdots \cap T^{-n}[i_{n}]\). By definition, for each \(n \in \mathbb{N}^{+}\) and each \([i_{0} \cdots i_{n}]\), we have \(T^{n}\) is a diffeomorphism from \([i_{0} \cdots i_{n}]\) to \([i_{n}]\). In particular, an element \(y\) in \(\mathcal{H}_{n}\) which is the set of inverse branches of \(T^{n}\) has the form \(y := (T^{n} |_{[i_{0} \cdots i_{n}]})^{-1}\). We say a function $h\in C^{Lip}(I)$ is not cohomologous to locally constant functions if there is no $\chi$ which is constant on each $[ij]$ and $u\in C^{Lip}(I)$ such that $h=\chi+u\circ T-u$. 

\begin{Lemma}\label{Lemma 5.1}
There exist $D>0$ and $x_{1}\not=x_{2}$ belong to some $I_{i}$ such that for arbitrarily large $n\in\mathbb{N}$ there exist two inverse branches $y_{1}\not=y_{2}\in\mathcal{H}_{n}$ such that 
$$|R_{y_{1},y_{2}}(x_{1})-R_{y_{1},y_{2}}(x_{2})|\ge D,$$
where $R_{y_{1},y_{2}}=r_{n}\circ y_{1}-r_{n}\circ y_{2}$.
\end{Lemma}
\begin{proof}
We prove this by contradiction. If not, then for any $D>0$, any $I_{i}$ and any $x_{1}\not=x_{2}\in I_{i}$ we have for any large enough $N$ and any two inverse branches $y^{\prime\prime}\not=y^{\prime}\in\mathcal{H}_{N}$,
\begin{equation}\label{5.1}
	|R_{y^{\prime\prime},y^{\prime}}(x_{1})-R_{y^{\prime\prime},y^{\prime}}(x_{2})|\le D|x_{1}-x_{2}|.
\end{equation}
We will prove the above implies $r$ is cohomologous to locally constant functions. 

Let $y^{\infty}$ be a set of countably many inverse branches of $\mathcal{H}_{1}$, namely $y^{\infty}=\{y_{j}\in\mathcal{H}_{1}\}_{j=1}^{\infty}$ where $y_{j+1}\circ y_{j}$ is well defined. We then define $y^{(n)}:=y_{n}\circ\cdots\circ y_{1}\in\mathcal{H}_{n}$. For each $I_{i}$, we fix a point $x_{i}\in I_{i}$. We define a function
$$h_{y^{\infty}}(x):=\sum_{n=1}^{\infty}r\circ y^{(n)}(x)-r\circ y^{(n)}(x_{i}),\quad x\in I_{i}.$$
By \eqref{4.1}, it is easy to show $h_{y^{\infty}}\in C^{Lip}(I)$. We claim that $r-h_{y^{\infty}}\circ T+h_{y^{\infty}}$ is locally constant which will complete the proof. Thus, it suffices to show that $r\circ y-h_{y^{\infty}}+h_{y^{\infty}}\circ y$ is constant on each $I_{i}$ for any $y\in\mathcal{H}_{1}$. By definition,
$$r\circ y-h_{y^{\infty}}+h_{y^{\infty}}\circ y=\sum_{n=1}^{\infty}r\circ v^{(n)}-r\circ y^{(n)}$$
where $v^{(1)}=y$ and $v^{(n)}=y^{(n-1)}\circ y$ if $k\ge2$. The key point is that, for any $x_{1},x_{2}\in I_{i}$, we have 
$$
\begin{aligned}
	&\bigg(\sum_{n=1}^{\infty}r\circ v^{(n)}(x_{1})-r\circ y^{(n)}(x_{1})\bigg)-\bigg(\sum_{n=1}^{\infty}r\circ v^{(n)}(x_{2})-r\circ y^{(n)}(x_{2})\bigg)\\
	=&\lim_{N\to\infty}R_{v^{(N)},y^{(N)}}(x_{1})-R_{v^{(N)},y^{(N)}}(x_{2})=0
\end{aligned}
$$
where the limit is implied by \eqref{5.1} with $y^{\prime\prime}=v^{(N)}$ and $y^{\prime}=y^{(N)}$. Therefore, we proved that $r\circ y-h_{y^{\infty}}+h_{y^{\infty}}\circ y$ is constant on each $I_{i}$ for any $y\in\mathcal{H}_{1}$. In particular, $r-h_{y^{\infty}}\circ T+h_{y^{\infty}}$ is constant on each $[ij]=I_{i}\cap T^{-1}I_{j}\not=\emptyset$ which completes the proof.
\end{proof}

We initiate the proof of Lemma \ref{Lemma 4.4} by fixing a large $n\in\mathbb{N}$ which satisfies 
\begin{equation}\label{5.2}
	4C_{3}C_{8}^{\alpha}\lambda^{\alpha n}|x_{1}-x_{2}|^{\alpha}\pi^{\alpha-1}\le \dfrac{D^{\alpha}}{12},
\end{equation}
and satisfies the conclusion of Lemma \ref{Lemma 4.1}, namely there exist two inverse branches $y_{1}\not=y_{2}\in\mathcal{H}_{n}$ such that 
\begin{equation}\label{6.4.2}
	|R_{y_{1},y_{2}}(x_{1})-R_{y_{1},y_{2}}(x_{2})|\ge D.
\end{equation}
One should note that, since $r\in C^{Lip}(I)$, we have $R_{y_{1},y_{2}}\in C^{Lip}(I)$ as well. Furthermore, by \eqref{4.1}, we have $|R_{y_{1},y_{2}}|_{Lip}\le C_{13}$ for some uniform constant $C_{13}>0$. Without loss of generality, we assume $R_{y_{1},y_{2}}(x_{1})<R_{y_{1},y_{2}}(x_{2})$ and $x_{1}<x_{2}$. We divide the interval $[R_{y_{1},y_{2}}(x_{1}),R_{y_{1},y_{2}}(x_{2})]$ into subintervals of size $\pi/3|b|$, and by connectivity we can find finitely many points $\{x^{\prime}_{k}\}_{k=1}^{p}\in [x_{1},x_{2}]$ such that
\begin{equation}\label{6.4.3}
	x^{\prime}_{k}<x_{k+1}^{\prime}\quad\text{and}\quad R_{y_{1},y_{2}}(x_{k}^{\prime})=R_{y_{1},y_{2}}(x_{1})+\frac{k\pi}{3|b|}.
\end{equation}
By \eqref{6.4.2}, we must have $\#\{x_{k}^{\prime}\}_{k}\ge \frac{3D|b|}{\pi}$, i.e., $p\ge \frac{3D|b|}{\pi}$.

\begin{Claim}
	There exist $ \frac{3D|b|}{2\pi}$ numbers $x_{k}^{\prime}$, i.e., exists a subset $\{x_{k_{j}}^{\prime}\}_{j=1}^{ \frac{3D|b|}{2\pi}}$ such that $|x^{\prime}_{k_{j}+1}-x^{\prime}_{k_{j}}|\le\frac{\pi}{D|b|}|x_{2}-x_{1}|$.
\end{Claim}
\begin{proof}
	We prove this by contradiction. If not, then there exists a subset $\{x_{k_{j}}^{\prime}\}_{j=1}^{ \frac{3D|b|}{2\pi}}$ such that $|x^{\prime}_{k_{j}+1}-x^{\prime}_{k_{j}}|\ge\frac{\pi}{D|b|}|x_{2}-x_{1}|$. Note that $\{[x^{\prime}_{k},x^{\prime}_{k+1}]\}_{k=1}^{p}$ forms a partition of $[x_{1},x_{2}]$. In particular, we should have $$|x_{2}-x_{1}|\ge\sum_{j=1}^{\frac{3D|b|}{2\pi}}|x^{\prime}_{k_{j}+1}-x^{\prime}_{k_{j}}|\ge\frac{3D|b|}{2\pi}\frac{\pi}{D|b|}|x_{2}-x_{1}|=\frac{3}{2}|x_{2}-x_{1}|,$$
	which contradicts $x_{1}\not=x_{2}$.
\end{proof}

By \eqref{6.4.3} and the above claim, for each $1\le j\le\frac{3D|b|}{2\pi}$, we have
\begin{equation}\label{6.4.4}
	|x^{\prime}_{k_{j}+1}-x^{\prime}_{k_{j}}|\le\frac{\pi}{D|b|}|x_{2}-x_{1}|\quad\text{and}\quad R_{y_{1},y_{2}}(x_{k_{j}+1}^{\prime})=R_{y_{1},y_{2}}(x_{k_{j}}^{\prime})+\frac{\pi}{3|b|}.
\end{equation}

\begin{Lemma}\label{Lemma 5.2}
There exists $\delta_{6}\in(0,1)$ such that for any $h \in C^{\alpha}(I)$ satisfying $|h|_{\alpha} \le 2C_{3}|b|^{\alpha}|h|_{\infty}$ and each $1\le j\le\frac{3D|b|}{2\pi}$, there exists $w_{j}\in[x^{\prime}_{k_{j}+1}, x^{\prime}_{k_{j}}]$ such that $|\mathcal{L}_{ib}^{n}h(w_{j})|\le(1-\delta_{6})|h|_{\infty}$.
\end{Lemma}
\begin{proof}
We first deal with the easy case where there exists $x\in[x^{\prime}_{k_{j}+1}, x^{\prime}_{k_{j}}]$ such that $|h|\circ y_{1}(x)$ or $|h|\circ y_{2}(x)\le \frac{1}{2}|h|_{\infty}$. We then have
$$
|\mathcal{L}_{ib}^{n}h(x)|\le \sum_{y\in\mathcal{H}_{n}; y \neq y_{1}}e^{\varphi_{n}\circ y(x)}|h|\circ y(x)+e^{\varphi_{n}\circ y_{1}(x)}\dfrac{1}{2}|h|_{\infty}\le(1-2^{-1}e^{-n|\varphi|_{\infty}})|h|_{\infty},
$$
or
$$
|\mathcal{L}_{ib}^{n}h(x)|\le \sum_{y\in\mathcal{H}_{n}; y \neq y_{2}}e^{\varphi_{n}\circ y(x)}|h|\circ y(x)+e^{\varphi_{n}\circ y_{2}(x)}\dfrac{1}{2}|h|_{\infty}\le(1-2^{-1}e^{-n|\varphi|_{\infty}})|h|_{\infty}.
$$
Thus, we have $|\mathcal{L}_{ib}^{n}h(x)|\le(1-2^{-1}e^{-n|\varphi|_{\infty}})|h|_{\infty}$. In particular, we can choose $w_{j}=x$ and $\delta_{6}=2^{-1}e^{-n|\varphi|_{\infty}}$ to complete the proof of this case.

It remains to deal with the difficult case where $|h|\circ y_{1}(x)$ and $|h|\circ y_{2}(x)\ge \frac{1}{2}|h|_{\infty}$ for any $x\in[x^{\prime}_{k_{j}+1}, x^{\prime}_{k_{j}}]$. In this case, for any $x\in[x^{\prime}_{k_{j}+1}, x^{\prime}_{k_{j}}]$, we can write $h\circ y_{1}(x)=|h|\circ y_{1}(x)e^{i\theta_{h}\circ y_{1}(x)}$ and $h\circ y_{2}(x)=|h|\circ y_{2}(x)e^{i\theta_{h}\circ y_{2}(x)}$with $\theta_{h}\in C^{\alpha}(I,\mathbb{R})$. Since $|h|_{\alpha} \le 2C_{3}|b|^{\alpha}|h|_{\infty}$ and $|h|\circ y_{1}(x)$ and $|h|\circ y_{2}(x)\ge \frac{1}{2}|h|_{\infty}$ for any $x\in[x^{\prime}_{k_{j}+1}, x^{\prime}_{k_{j}}]$, it not difficult to obtain that
\begin{equation*}
	|\theta_{h}\circ y_{1}(x)-\theta_{h}\circ y_{1}(x^{\prime})|\le 4C_{3}|b|^{\alpha}|y_{1}(x)-y_{1}(x^{\prime})|^{\alpha},
\end{equation*}
and
\begin{equation*}
	|\theta_{h}\circ y_{2}(x)-\theta_{h}\circ y_{2}(x^{\prime})|\le 4C_{3}|b|^{\alpha}|y_{2}(x)-y_{2}(x^{\prime})|^{\alpha},
\end{equation*}
for any $x,x^{\prime}\in[x^{\prime}_{k_{j}+1}, x^{\prime}_{k_{j}}]$. By \eqref{4.1} and \eqref{5.2}, we then further bound the above two bounds as follows:
\begin{equation}\label{6.4.5}
	|\theta_{h}\circ y_{1}(x)-\theta_{h}\circ y_{1}(x^{\prime})|\quad\text{and}\quad|\theta_{h}\circ y_{2}(x)-\theta_{h}\circ y_{2}(x^{\prime})|\le \frac{D^{\alpha}|b|^{\alpha}\pi^{1-\alpha}}{12}\frac{|x-x^{\prime}|^{\alpha}}{|x_{1}-x_{2}|^{\alpha}},
\end{equation}
for any $x,x^{\prime}\in[x^{\prime}_{k_{j}+1}, x^{\prime}_{k_{j}}]$. Then, by \eqref{6.4.4} and \eqref{6.4.5}, we have
\begin{equation*}
	\dfrac{\pi}{6}\le|(bR_{y_{1},y_{2}}+\theta_{h}\circ y_{1}-\theta_{h}\circ y_{2})(x_{k_{j}+1}^{\prime})-(bR_{y_{1},y_{2}}+\theta_{h}\circ y_{1}-\theta_{h}\circ y_{2})(x_{k_{j}}^{\prime})|\le \dfrac{\pi}{2}.
\end{equation*}
As a corollary of the above estimate, it follows that
\begin{equation}\label{6.4.6}
	|(bR_{y_{1},y_{2}}+\theta_{h}\circ y_{1}-\theta_{h}\circ y_{2})(x_{k_{j}+1}^{\prime})-2\pi\mathbb{Z}|\ge \frac{\pi}{12},
\end{equation}
or
\begin{equation}\label{6.4.7}
	|(bR_{y_{1},y_{2}}+\theta_{h}\circ y_{1}-\theta_{h}\circ y_{2})(x_{k_{j}}^{\prime})-2\pi\mathbb{Z}|\ge \frac{\pi}{12}.
\end{equation}
Therefore, utilizing an elementary result from \cite[Lemma 4.1]{Zha24}, and combining \eqref{6.4.6} and \eqref{6.4.7}, for some uniform constant $\varepsilon_{7}\in(0,1)$ and $\sigma=x_{k_{j}}^{\prime}$ or $x_{k_{j}+1}^{\prime}$, we obtain:
\begin{equation}\label{6.4.8}
	\begin{aligned}
		&|e^{(\varphi-ibr)_{n}\circ y_{1}(x_{\sigma}^{\prime})}h\circ y_{1}(x_{\sigma}^{\prime})+e^{(\varphi-ibr)_{n}\circ y_{2}(x_{\sigma}^{\prime})}h\circ y_{2}(x_{\sigma}^{\prime})|\\
		\le&(1-\varepsilon_{7})e^{\varphi_{n}\circ y_{1}(x_{\sigma}^{\prime})}|h|_{\infty}+e^{\varphi_{n}\circ y_{2}(x_{\sigma}^{\prime})}|h|_{\infty},
	\end{aligned}
\end{equation}
or
\begin{equation}\label{6.4.9}
	\begin{aligned}
		&|e^{(\varphi-ibr)_{n}\circ y_{1}(x_{\sigma}^{\prime})}h\circ y_{1}(x_{\sigma}^{\prime})+e^{(\varphi-ibr)_{n}\circ y_{2}(x_{\sigma}^{\prime})}h\circ y_{2}(x_{\sigma}^{\prime})|\\
		\le&e^{\varphi_{n}\circ y_{1}(x_{\sigma}^{\prime})}|h|_{\infty}+(1-\varepsilon_{7})e^{\varphi_{n}\circ y_{2}(x_{\sigma}^{\prime})}|h|_{\infty}.
	\end{aligned}
\end{equation}
In particular, this implies $|\mathcal{L}^{n}_{ib}h(x_{k_{j}}^{\prime})|$ or $|\mathcal{L}^{n}_{ib}h(x_{k_{j}+1}^{\prime})|\le(1-\varepsilon_{7}e^{-n|\varphi_{a}|_{\infty}})|h|_{\infty}$. Therefore, in this case, we can choose $\delta_{6}=\varepsilon_{7}e^{-n|\varphi|_{\infty}}$ and $w_{j}=x_{k_{j}}^{\prime}$ or $x_{k_{j}+1}^{\prime}$ to complete the proof.
\end{proof}

By Lemma \ref{Lemma 4.1}, provided $n$ is large sufficiently, it is not difficult to show that $|\mathcal{L}^{n}_{ib}h|_{\alpha}\le 2C_{3}|b|^{\alpha}|h|_{\infty}$ for any $h \in C^{\alpha}(I)$ with $|h|_{\alpha} \le 2C_{3}|b|^{\alpha}|h|_{\infty}$. Then, by Lemma \ref{Lemma 5.2}, for each $1\le j\le\frac{3D|b|}{2\pi}$ and any $|x-w_{j}|\le \frac{\varepsilon_{6}}{4C_{3}|b|}$, we have $|\mathcal{L}^{n}_{ib}h(x)|\le(1-\varepsilon_{6}/2)|h|_{\infty}$. Since $\mu_{\varphi}$ is the SRB measure of $T:I\to I$, the subset $B:=\bigsqcup_{j=1}^{\frac{3D|b|}{2\pi}}B(w_{j}, \frac{\varepsilon_{6}}{4C_{3}|b|})$ has a uniformly lower bounded $\mu_{\varphi}$-measure which completes the proof of Lemma \ref{Lemma 4.4}.

\begin{Remark}\label{Remark on rapid mixing}
If the roof function \( r \) is \( C^{\alpha} \) for some \( \alpha \in (0,1) \) and is not cohomologous to locally constant functions, then it is possible to show that the semiflow \( \phi_{t}: I_{r} \to I_{r} \) is rapidly mixing for any Gibbs measure. Indeed, it is straightforward to observe that the weaker UNI condition derived in Lemma \ref{Lemma 5.1} remains valid when the roof function \( r \) is only H\"older continuous. Then, one can adjust the proof of Lemma \ref{Lemma 4.4} to establish a modest cancellation estimate on a small set. By combining this argument with techniques from \cite{Zha24}, it is feasible to demonstrate that the semiflow \( \phi_{t}: I_{r} \to I_{r} \) exhibits rapid mixing for any Gibbs measure.
\end{Remark}

\begin{Remark}
	In this paper, we focus solely on flows and semiflows with uniform hyperbolicity, specifically suspension semiflows of one-dimensional expanding maps. Our main objective is to provide a mechanism for stretched-exponential mixing in a setting that is as simple as possible without losing generality. However, with appropriate adjustments, it becomes evident that our argument can be extended to suspension semiflows of Young towers. 
\end{Remark}

\begin{Remark}
	It is interesting to explore whether the argument presented here can be generalized to systems with discontinuities. For instance, the examples considered in \cite{Esl17}, where the author proved stretched-exponential mixing for circle extensions of one-dimensional expanding maps with discontinuities under the assumption that the skew function is \(C^{1}\). It is possible to weaken the \(C^{1}\) setting to a Lipschitz setting.
\end{Remark}

 \end{document}